%% file: LouderPerinSklinos.tex
\documentclass[a4paper,11pt]{article}

\usepackage[english]{babel}   
\usepackage[latin1]{inputenc}   
\usepackage[T1]{fontenc}   
    
\usepackage{graphicx}
\usepackage{vmargin}
\usepackage{color}

\usepackage{latexsym}
\usepackage{amsmath}
\usepackage{amssymb}
\usepackage{mathrsfs}            
\usepackage{amsthm}             
\usepackage{amsfonts}                     
\usepackage{float}              
 
\usepackage{hyperref}

\setmarginsrb{2.7cm}{2.5cm}{2.7cm}{2.6cm}{0.5cm}{1cm}{0.5cm}{0.8cm}

\def\fg{\unskip~{\fontencoding{T1}\selectfont\guillemotright}\space}

\newtheoremstyle{th}
{4pt}{5pt}      
{\it}           
{}              
{\bf }          
{:}             
{.5em}          
{}              

\theoremstyle{th} 
\newtheorem{theorem}{Theorem}[section]

\newtheorem{definition}[theorem]{Definition}

\newtheorem{remark}[theorem]{Remark}

\newtheorem{corollary}[theorem]{Corollary}
\newtheorem{lemma}[theorem]{Lemma}

\newtheorem{proposition}[theorem]{Proposition}

\newtheorem{question}{Question}

\newtheoremstyle{ex}
{4pt}{5pt}      
{}              
{}              
{\bf }          
{:}             
{.5em}          
{}              

\theoremstyle{ex}

\newcommand{\Z}{\mathbb{Z}}

\newcommand{\F}{\mathbb{F}}
\newcommand{\M}{\mathbb{M}}

\newcommand{\abs}[1]{\ensuremath{\left|#1\right|}}

\newcommand{\Aut}{\ensuremath{\mathrm{Aut}}}

\newcommand{\tp}{\ensuremath{\mathrm{tp}}}

\date{}

\author{Larsen Louder, Chlo\'e Perin and Rizos Sklinos}
\title{Hyperbolic towers and independent generic sets in the theory of free groups}

\begin{document}

\maketitle

\begin{abstract} We use hyperbolic towers to answer some model theoretic questions around the generic type in the theory of free groups. 
We show that all the finitely generated models of this theory realize the generic type $p_0$, but that there is a finitely generated model which 
omits $p^{(2)}_0$. We exhibit a finitely generated model in which there are two maximal independent sets of realizations of the generic type which 
have different cardinalities. We also show that a free product of homogeneous groups is not necessarily homogeneous.
\end{abstract}

\begin{center}\textit{To Anand Pillay on the occasion of his 60th birthday.}\end{center}

\section{Introduction}\label{IT}

This paper is motivated by the works of Pillay \cite{PillayGenericity} and of the third named author \cite{SklinosGenericType}, 
which study the weight of the generic type in the free group.

Following the work of Sela \cite{Sel1}-\cite{Sel6} and of Kharlampovich and Myasnikov proving that non abelian 
free groups are elementarily equivalent, we denote by $T_{fg}$ their common first order theory. Sela also showed that $T_{fg}$ is stable \cite{SelaStability}.

Every stable theory admits a good model theoretic notion of independence, of which we give a brief account in Section \ref{Stable} for readers lacking a model theory background (the interested reader is referred to \cite{PillayStability}).

Poizat proved that $T_{fg}$ is connected in the sense of model theory, i.e. that there is a model of $T_{fg}$ which admits 
no proper definable subgroup of finite index. A consequence of stability is that connectedness is equivalent to saying that $T_{fg}$ admits a 
unique generic type over any set of parameters. We denote by $p_0$ the generic type over the empty set. 
Pillay gives a characterization of elements that realize $p_0$ in non abelian free groups. In fact, he shows more generally:

\begin{theorem} \cite[Fact 1.10(ii) and Theorem 2.1(ii)]{PillayGenericity}  \label{MaxIndAreBases}
A subset of a non abelian free group $\F$ is a maximal independent set of realizations of $p_0$ if and only if it is a basis of $\F$.
\end{theorem} 

An immediate consequence of this result is that in a non abelian free group $\F$, maximal independent sets of realizations of the 
generic type all have the same cardinality.

The notion of weight of a type $p$ can be intuitively thought of as a generalized exchange principle (see Section \ref{Stable}), 
and when it is finite, it bounds the ratio of the cardinality of two maximal independent sets of realizations of $p$. 
In particular, it is straightforward to see that if a type $p$ has weight $1$, then any two maximal 
independent sets of realizations of $p$ (in any model) have the same cardinality. 

In this light, the following result might look surprising:

\begin{theorem} \cite{PillayGenericity, SklinosGenericType}
The generic type $p_0$ of the theory of non abelian free groups has infinite weight.  
\end{theorem}

Intuitively one can explain this behavior of the generic type by noticing that 
two bases of a fixed non abelian free group have the same cardinality as a consequence of the universal property 
and not of some exchange principle.

It is thus natural to ask whether we can witness infinite weight in an explicit model in terms of independent 
sets of generic elements, or even whether we can witness that the generic type does not have weight $1$:

\begin{question} \label{MaxIndepGenericsQ} Is there a model of the theory of the free group in which one can find two 
maximal independent sets of realizations of the generic type with different cardinalities? 
\end{question}

Sections \ref{Stable} and \ref{Gmt} serve as introductory material for the notions we need from model theory and geometric group theory respectively. In 
Section \ref{Stable} we give formal definitions of independence and weight. 
In Section \ref{Gmt} we describe in detail the geometric notion 
which lies at the core of this paper, namely hyperbolic towers. 

In Section \ref{IsolatingResultsSec}, we examine $p_0$ with respect to the notion of isolation. We give the proof of an unpublished result of 
Pillay which shows that $p^{(2)}_0$ is not isolated in the theory axiomatized by $p_0$ (after adding a constant to the language of groups).We then classify the hyperbolic tower structures admitted by the
fundamental group $S_4$ of the connected sum of four projective
planes. We use this to deduce that the type $p_0$ is realized in every
finitely generated model of $T_{fg}$, but that $S_4$ omits
$p^{(2)}_0$, thus giving an explicit witness to Pillay's non isolation
result.  This also enables us to see that no type in $S(T_{fg})$
(apart from the trivial one) is isolated.

In Section \ref{MaxIndepSeqSec}, we answer Question \ref{MaxIndepGenericsQ} in the affirmative by exhibiting a suitable 
finitely generated model of $T_{fg}$. 

Finally we use this result in Section \ref{HomOfFreeProductsSec} to see that homogeneity is not preserved under taking 
free products, thus answering a question of Jaligot.

\paragraph{Acknowledgements.} We wish to thank Anand Pillay for allowing us to include his proof of Theorem \ref{P0squaredOverP0}, as well as 
Frank Wagner for some useful conversations. We are grateful to the referee for their thorough reading and numerous comments.

Finally, we would like to thank the organizers of the conference ``Recent Developments in Model Theory'' at Ol\'eron.

\section{Independence and Weight}\label{Stable}

In this section we give a quick description of the model theoretic notions we use. The exposition is biased towards our 
needs and by no means complete. 

We fix a stable first order theory $T$ and we work in a ``big'' 
saturated model $\mathbb{M}$, which is usually called the {\em monster model} (see \cite[p.218]{MarkerModelTheory}). 
As mentioned in the introduction, stable theories admit a good notion of independence, 
the prototypical examples being linear independence in vector spaces and algebraic independence in algebraically closed fields. 

In a more abstract setting Shelah gave the following definition of (forking) independence. 

\begin{definition}
Let $\phi(\bar{x},\bar{b})$ be a first order formula in $\mathbb{M}$ and $A\subset M$ (the underlying domain of $\mathbb{M}$). Then 
$\phi(\bar{x},\bar{b})$ forks over $A$ if there are infinitely many automorphisms $(f_i)_{i<\omega}\in Aut_A(\mathbb{M})$ and 
some $k<\omega$, such that the set $\{\phi(\bar{x},f_i(\bar{b})): i<\omega\}$ is $k$-inconsistent, i.e. every subset 
of cardinality $k$ is inconsistent.
\end{definition}

Recall that an $m$-type $p(\bar{x})$ over $A \subseteq M$ of the first order theory $T$ is a consistent (with $T$) set of formulas with parameters in 
$A$ with at most $m$ free variables. For example the type $\tp(\bar{a}/A)$ of a tuple $\bar{a} \in M$ is the set of formulas that $\bar{a}$ satisfies in $\mathbb{M}$ 
(in fact, saying that $\mathbb{M}$ is saturated is exactly saying that every $m$-type $p(\bar{x})$ over a set of parameters of cardinality 
strictly less than $\abs{M}$ is the type of an $m$-tuple $\bar{a} \in M$). 
The type of $\bar{a}$ over $A$ can equivalently be thought of as the collection of sets which are definable over $A$ and which contain $\bar{a}$. 

If $A\subseteq B$, we say that $\bar{a}$ is {\em independent} from $B$ over $A$ if there is no formula in $\tp(\bar{a}/B)$ which forks over $A$. 
In the opposite case we say that $\bar{a}$ {\em forks} with $B$ over $A$ (or $\bar{a}$ is not independent from $B$ over $A$). 
Heuristically, one can think of the latter case as expressing the fact that the type of $\bar{a}$ over 
$B$ contains much more information than the type of $\bar{a}$ over $A$ alone. 

Indeed, the definition implies that there is a formula with parameters in $B$ satisfied by $\bar{a}$ which forks over $A$. 
Thus the set $X$ defined by this formula contains $\bar{a}$, is definable over $B$, and admits an infinite sequence of 
$k$-wise disjoint translates by elements of $\Aut_A(\M)$ (here $\Aut_A(\M)$ denotes automorphisms of $\M$ fixing $A$ pointwise).

Consider now a set $Y$ which is definable over $A$ alone and contains $\bar{a}$: 
any automorphism in $\Aut_A(\M)$ necessarily fixes $Y$. Clearly $X$ can be assumed to be contained in $Y$, 
and thus so are all of its automorphic images (under $Aut_A(\M)$). Thus in some sense, $X$ is much smaller 
than any definable set $Y$ given by a formula in $tp(\bar{a}/A)$, and the type of $\bar{a}$ over $B$ ``locates'' 
$\bar{a}$ much more precisely than its type over $A$ alone.

A consequence of stability is the existence of non forking (independent) extensions. Let 
$A\subseteq B$ and $p(\bar{x})$ be a type over $A$. Then we say that $q(\bar{x}):=tp(\bar{a}/B)$ is a {\em non-forking extension} of $p(\bar{x})$, 
if $p(\bar{x})\subseteq q(\bar{x})$ and moreover $\bar{a}$ does not fork with $B$ over $A$. A type over $A$ is called {\em stationary} 
if for any $B\supseteq A$ it admits a unique non forking extension over $B$.

Let $C=\{\bar{c}_i : i\in I\}$ be a set of tuples, we say that {\em $C$ is an independent set over $A$} if for 
every $i\in I$, $\bar{c}_i$ is independent from $A \cup C\setminus\{\bar{c}_i\}$ over $A$. If $p$ is a type over $A$ 
which is stationary and  $(a_i)_{i<\kappa}, (b_i)_{i<\kappa}$ are both independent sets over $A$ of realizations of $p$, 
then $tp((a_i)_{i<\kappa}/A)=tp((b_i)_{i<\kappa}/A)$ (see \cite[Lemma 2.28, p.29]{PillayStability}). 
This allows us to denote by $p^{(\kappa)}$ the type of $\kappa$-independent realizations of $p$. 

For the purpose of assigning a dimension (with respect to forking independence) to a type, one might 
ask what is the cardinality of a maximal independent set of realizations of a type and whether any 
two such sets have the same cardinality. This naturally leads to the definition of weight. 

\begin{definition}
The preweight of a type $p(\bar{x}):= tp(\bar{a}/A)$, $prwt(p)$ is the
supremum of the set of cardinals $\kappa$ for which there exists $\{\bar{b}_i : i <\kappa\}$ an independent set over $A$, such that $\bar{a}$ forks with $\bar{b}_i$ over $A$ for all $i$.
The weight $wt(p)$ of a type $p$ is the supremum of 
$$\{prwt(q) \mid q \mathrm{\; a \; non \; forking \; extension \; of \; } p\}.$$
\end{definition}

The special case of weight $1$ can be thought of as an exchange principle: an element $a$ in the set of 
realizations of a weight $1$ type cannot fork with more than one element from an independent set. 

Thus, as in the case of the dimension theorem for vector spaces one can easily see that any two maximal 
independent sets of realizations of a weight $1$ type must have the same cardinality.

\section{Hyperbolic Towers}\label{Gmt}

In this section we define hyperbolic towers. Hyperbolic towers 
were first used by Sela \cite{Sel6} to describe the finitely generated models of the theory of non abelian free groups. 
They also appeared in \cite{PerinElementary} where the geometric structure of a group 
that elementarily embeds in a torsion-free hyperbolic group is characterized. 

In order to define hyperbolic towers we need to give a few preliminary definitions.

\subsection{Graphs of groups and graphs of spaces}

We first go briefly over the notion of graph of groups, for a more formal definition and further properties the reader is referred to \cite{SerreTrees}.

 A {\em graph of groups} consists of a graph $\Gamma$, together with two collections of groups $\{G_v\}_{v \in V(\Gamma)}$ 
 (the {\em vertex groups}) and $\{G_e\}_{e \in E(\Gamma)}$ (the {\em edge groups}), and a collection of embeddings $G_e \hookrightarrow G_v$ for 
 each pair $(e, v)$ where $e$ is an edge and $v$ is one of its endpoints. To a graph of groups $\Gamma$ is associated a group $G$ called 
 its {\em fundamental group} and denoted $\pi_1(\Gamma)$ (the use of algebraic topology terminology will be made clear below). 
 There is a canonical action of this group $G$ on a simplicial tree $T$ whose quotient $G \backslash T$ is isomorphic to $\Gamma$. 
 Conversely, to any action of a group $G$ on a simplicial tree $T$ without inversions, one can associate a graph of groups $\Gamma$ whose 
 fundamental group is isomorphic to $G$ and whose underlying graph is isomorphic to the quotient $G \backslash T$. An element or a subgroup in 
 $G$ which fixes a point in $T$ (or equivalently, which is 
contained in a conjugate in $G$ of one of the vertex groups $G_v$) is said to be {\em elliptic}.

A fundamental example is the special case where $\Gamma$ consists of two vertices $v$ and $w$ joined by a single edge $e$: then, the fundamental 
group of $\Gamma$ is the {\em amalgamated product} $G_v * _{G_e} G_w$. Graphs of groups can thus be thought of as a generalized version of 
amalgamated products.

The Van Kampen lemma gives a useful perspective on graphs of groups. It states that if a topological space $X$ can be written as a union 
$X_1 \cup X_2$ of two of its path connected subspaces, and if $Y = X_1  \cap X_2$ is also path connected, the (usual) fundamental group $\pi_1(X)$ of the space $X$ 
can be written as an amalgamated product $\pi_1(X_1) *_{\pi_1(Y)} \pi_1(X_2)$ where the group embeddings $\pi_1(Y) \hookrightarrow \pi_1(X_i)$ 
are induced by the topological embeddings $Y \hookrightarrow X_i$. 

Similarly, to a graph of groups $\Gamma$ we can associate a (not unique) {\em graph of spaces}: to each vertex $v \in V(\Gamma)$ 
(respectively edge $e \in E(\Gamma)$), we associate a (sufficiently nice) topological space $X_v$ (respectively $X_e$) 
such that $\pi_1(X_v) = G_v$ (respectively $\pi_1(X_e) = G_e$). To each pair $(e,v)$ of an edge and an 
endpoint is associated a topological embedding $f_{e,v} : X_e \hookrightarrow X_v$ which induces on 
fundamental groups the embedding $G_e \hookrightarrow G_v$. Then the fundamental group of the graph of groups $\Gamma$ is 
isomorphic to the fundamental group $\pi_1(X)$ of the space $X$ built by gluing the collection of spaces $\{X_v \mid v \in V(\Gamma)\}$ 
and $\{X_e \times [0,1] \mid e \in E(\Gamma)\}$ using the maps $f_{e,v}$. More precisely if $e$ is an edge joining $v$ to $w$, 
we identify each point $(x,0)$ of $X_e \times \{0\}$ to the image of $x$ in $X_v$ under $f_{e,v}$ and each point $(x, 1)$ in $X_e \times \{1\}$ to the image of $x$ in $X_w$ under $f_{e,w}$. 
Conversely, given a graph of spaces, there is a graph of groups associated to it. Figure \ref{GOGFig} illustrates this duality.

\begin{figure}[ht] 
\centering
\scalebox{1}{
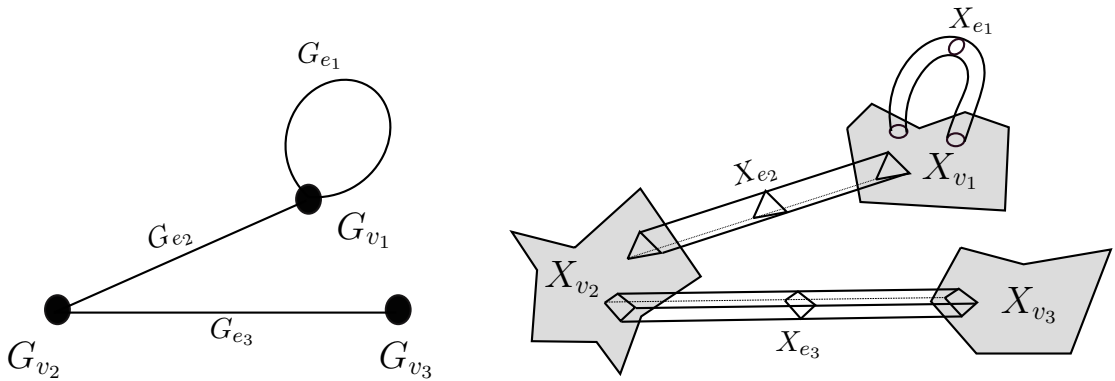}
\caption{A graph of groups and an associated graph of spaces.}
\label{GOGFig}
\end{figure}

\begin{definition} \emph{(Bass-Serre presentation)} Let $G$ be a group acting on a simplicial tree $T$ without inversions,
denote by $\Gamma$ the corresponding quotient graph of groups and by $p$ the quotient map $T \to \Gamma$. A Bass-Serre presentation for $\Gamma$ is a pair $(T^1, T^0)$ consisting of
\begin{itemize}
\item a subtree $T^1$ of $T$ which contains exactly one edge of $p^{-1}(e)$ for each edge $e$ of $\Gamma$;
\item a subtree $T^0$ of $T^1$ which is mapped injectively by $p$ onto a maximal subtree of $\Gamma$;
\end{itemize}
\end{definition}
The choice of terminology is justified by the fact that to such a pair $(T^1, T^0)$, we can associate a presentation of $G$ in terms of the subgroups $G_v$ for $v \in V(T^0)$, and elements of $G$ which send vertices of $T^1$ in $T^0$ (Bass-Serre generators).

\subsection{Surface groups}

We now recall some standard facts about surfaces and surface groups. Unless otherwise mentioned, all surfaces are assumed to be compact and connected.

The classification of surfaces gives that a surface without boundary $\Sigma$ (or {\em closed surface}) is characterized up to homeomorphism by its orientability and its Euler characteristic $\chi(\Sigma)$. It can be easily deduced from this that a surface with (possibly empty) boundary $\Sigma$ is characterized up to homeomorphism by its orientability, its Euler characteristic $\chi(\Sigma)$ and the number of its boundary components. The orientable closed surface of characteristic $2$ is the sphere, that of characteristic $0$ is the torus; the non orientable closed surface of characteristic $1$ is the projective plane, and that of characteristic $0$ is the Klein bottle. 

The {\em connected sum} $\Sigma$ of two surfaces $\Sigma_1$ and $\Sigma_2$ is the surface obtained by removing an open disk from each $\Sigma_i$, and gluing the boundary components thus obtained one to the other. We then have $\chi(\Sigma)= \chi(\Sigma_1)+\chi(\Sigma_2) - 2$. One then sees for example that the closed non orientable surface of characteristic $-1$ is the connected sum of three projective planes. Puncturing a surface (i.e. removing an open disk) decreases the Euler characteristic by $1$.

Let $\Sigma$ be a surface with boundary. Each connected component of $\partial \Sigma$ has cyclic fundamental group, and gives rise in $\pi_1(\Sigma)$ to a 
conjugacy class of cyclic subgroups, which we call {\em maximal boundary subgroups}. A {\em boundary subgroup} of $\pi_1(\Sigma)$ is a non trivial subgroup of a maximal boundary subgroup of $\pi_1(\Sigma)$.

Suppose $\Sigma$ has $r$ boundary components, and let $\gamma_1, \ldots, \gamma_r$ be generators of non conjugate maximal boundary subgroups. Then $\pi_1(\Sigma)$ admits a presentation of the form
$$ \langle a_1, \ldots, a_{2m}, \gamma_1, \ldots, \gamma_r \mid [a_1, a_2]\ldots[a_{2m-1},a_{2m}] = \gamma_1 \ldots \gamma_r \rangle$$ 
if it is orientable, and
$$ \langle d_1, \ldots, d_q, \gamma_1, \ldots, \gamma_r \mid d^2_1 \ldots d_q^2= \gamma_1 \ldots \gamma_r \rangle$$ 
if not. The Euler characteristic of the corresponding surface is given by $-(2m-2+r)$ in the orientable case and $-(q-2+r)$ in the non orientable case.

Note that in particular, the fundamental group $\pi_1(\Sigma)$ of a compact surface $\Sigma$ with non empty boundary $\partial \Sigma$ is a free group. However, we think of it as endowed with the peripheral structure given by its collection of maximal boundary subgroups.

Note also that the presentation given for the non orientable case is equivalent to
$$ \langle a_1, \ldots, a_{2h}, d_1, \ldots, d_p, \gamma_1, \ldots, \gamma_r \mid [a_1, a_2]\ldots[a_{2h-1},a_{2h}]d^2_1 \ldots d_p^2= \gamma_1 \ldots \gamma_r \rangle$$
for any $h, p$ such that $2h+p=q$. This comes from the fact that the $r$-punctured connected sum of $h$ tori and $p$ projective planes (for $p>0$) is homeomorphic to the $r$-punctured connected sum of $2h+p$ projective planes (since they are both non orientable, have the same Euler characteristic, and $r$ boundary components).

Let $S$ be the fundamental group of $\Sigma$ a surface with boundary, and let ${\cal C}$ be a set of $2$-sided disjoint simple closed curves on $\Sigma$. Let $\{T_c \mid c \in {\cal C}\}$ be a collection of disjoint open neighborhoods of the curves of ${\cal C}$ with homeomorphisms $c \times (-1,1)  \to T_c$ sending $c \times \{0\}$ onto $c$. Then $\Sigma$ can be seen as a graph of spaces, with edge spaces the curves in ${\cal C}$, and vertex spaces the connected components of $\Sigma - \bigcup_{c \in {\cal C}} T_c$. This gives a graph of groups decomposition for $S$, in which edge groups are infinite cyclic and boundary subgroups are elliptic. Such a decomposition is called the {\em decomposition of $S$ dual to ${\cal C}$}. The following lemma gives a useful converse, it is essentially Theorem III.2.6 in \cite{MorganShalen}.
\begin{lemma} \label{GOGSurfaceGroup}
Let $S$ be the fundamental group of a surface with boundary $\Sigma$. Suppose $S$ admits a graph of groups decomposition $\Gamma$ in which edge groups are cyclic and boundary subgroups are elliptic. Then there exists a set of disjoint simple closed curves on $\Sigma$ such that $\Gamma$ is the graph of groups decomposition dual to ${\cal C}$.
\end{lemma}

The idea of the proof of this lemma is to build an $S$-equivariant map $f$ between a universal cover $\tilde{\Sigma}$ of $\Sigma$ and the tree $T$ associated to $\Gamma$, and to consider the preimage by $f$ of midpoints of edges of $T$. If $f$ is suitably chosen, this preimage will be the lift of a collection of simple closed curves ${\cal C}$ we are looking for.

\subsection{Hyperbolic floors and towers}

We will be interested in graphs of groups in which some of the vertex groups are {\em surface groups}, that is, fundamental groups of surfaces 
(all surfaces will be compact and with possibly non empty boundary). 
Equivalently, this means that the corresponding graph of spaces will have subspaces $X_v$ which are surfaces.

\begin{definition}
A graph of groups with surfaces is a graph of groups $\Gamma$ together with a subset 
$V_{S}$ of the set of vertices $V(\Gamma)$ of $\Gamma$, such that any vertex $v$ in $V_S$ satisfies:
\begin{itemize}
\item there exists a compact connected surface $\Sigma$ with non empty boundary, such that the vertex group $G_v$ is the fundamental group 
      $\pi_1(\Sigma)$ of $\Sigma$;
\item for each edge $e$, and $v$ an endpoint of $e$, the injection $G_e \hookrightarrow G_v$ maps $G_e$ onto a maximal 
      boundary subgroup of $\pi_1(\Sigma)$;
\item this induces a bijection between the set of edges adjacent to $v$ and the set of conjugacy classes in 
       $\pi_1(\Sigma)$ of maximal boundary subgroups of $\pi_1(\Sigma)$.
\end{itemize}
The vertices of $V_S$ are called surface type vertices. The surfaces associated to the vertices of $V_S$ are called the surfaces of $\Gamma$.
\end{definition}

Figure \ref{GOGWSFig} gives an example of a graph of groups with surfaces. Each surface type vertex $v$ of $\Gamma$ has been replaced 
by a picture of the corresponding surface with boundary $\Sigma_v$. Note how we represent pictorially the property that each edge group $G_e$ 
adjacent to a surface type vertex group $G_v$ embeds in a maximal boundary subgroup of $G_v$.

\begin{figure}[ht] 
\centering
\scalebox{0.7}{
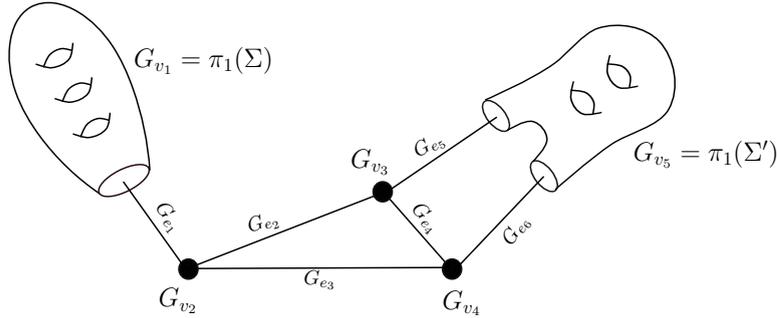}
\caption{A graph of groups with surfaces.}
\label{GOGWSFig}
\end{figure}

\begin{definition} \emph{((extended) hyperbolic floor)} \label{HypFloor} Consider a triple $(G, G', r)$ where $G$ is a group, 
$G'$ is a subgroup of $G$, and $r$ is a retraction from $G$ onto $G'$ (i.e. $r$ is a morphism $G \to G'$ which restricts to the identity on $G'$).

We say that $(G, G', r)$ is an extended hyperbolic floor if there exists a non trivial decomposition $\Gamma$ of 
$G$ as a graph of groups with surfaces, and a Bass-Serre presentation $(T^1, T^0)$ of $\Gamma$ such that:
\begin{itemize}
\item the surfaces of $\Gamma$ which are not once punctured tori have Euler characteristic at most $-2$; 
\item $G'$ is the free product of the stabilizers of the non surface type vertices of $T^0$;
\item every edge of $\Gamma$ joins a surface type vertex to a non surface type vertex (bipartism);
\item either the retraction $r$ sends surface type vertex groups of $\Gamma$ to non abelian images; 
or $G'$ is cyclic and there exists a retraction $r': G * \Z \to G' * \Z$ which sends surface type vertex groups of $\Gamma$ to non abelian images.
\end{itemize}
If the first alternative holds in this last condition, we say that $(G,G',r)$ is a hyperbolic floor.
\end{definition}

\begin{definition} \emph{((extended) hyperbolic tower)} \label{HypTower}
Let $G$ be a non cyclic group, let $H$ be a subgroup of $G$.
We say that $G$ is an (extended) hyperbolic tower over $H$ if there exists a finite
sequence $G=G^0 \geq G^1 \geq \ldots \geq G^m \geq H$ of subgroups of $G$ where $m \geq 0$ and:
\begin{itemize}
\item for each $k$ in $[0, m-1]$, there exists a  retraction $r_k:G^{k} \rightarrow G^{k+1}$
such that the triple $(G^k, G^{k+1}, r_k)$ is an (extended) hyperbolic floor, and $H$ is contained in one of the non surface type vertex group 
of the corresponding hyperbolic floor decomposition;
\item $G^m = H * F * S_1 * \ldots * S_p$ where $F$ is a (possibly trivial) free group, $p \geq 0$, and each $S_i$
is the fundamental group of a closed surface without boundary of Euler characteristic at most $-2$.
\end{itemize}
\end{definition}
Note that all the floors $(G^k, G^{k+1}, r_k)$ are in fact (non extended) hyperbolic floors except possibly for $(G^{m-1}, G^m, r_{m-1})$, 
and in this case $G^m$ is infinite cyclic so $H$ is cyclic or trivial. In particular, extended hyperbolic towers over non abelian subgroups are in fact hyperbolic towers.

\begin{figure}[!ht]
\begin{center}
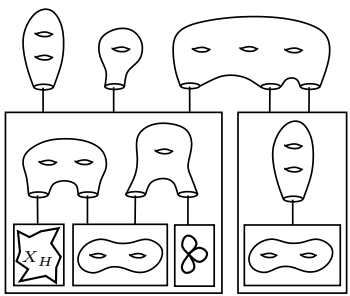
\caption{A hyperbolic tower over $H$.}
\label{HypTowerEx}
\end{center}
\end{figure}

To understand them better, let us consider hyperbolic towers in a graph of space perspective. If $G$ is an extended hyperbolic tower over $H$, 
it means we can build a space $X_G$ with fundamental group $G$ from a space $X_H$ with fundamental group $H$ as follows: start 
with the disjoint union $X^m$ of $X_H$ with closed surfaces $\Sigma_1, \ldots, \Sigma_p$ of Euler characteristic at most $-2$, together with a graph $X_F$. 
When $X^k$ is built, glue surfaces with boundary to $X^k$ along their boundary components (gluing each boundary component to non null 
homotopic curve of $X^k$) to obtain the space $X^{k-1}$. 

This is represented in Figure \ref{HypTowerEx}: here, $X^m$ is the union of the spaces in the four small boxes, and $X^{m-1}$ 
the union of those in the two big square boxes. Finally, $G$ is the fundamental group of the whole space 
(an edge between a surface and a box indicates that the corresponding boundary components is glued to a curve in the space contained in the box). 
In addition, one should think that each surface retracts onto the lower floor, in a non abelian way in the case of hyperbolic floors.

Though hyperbolic towers were introduced by Sela, their definition was slightly too restrictive, and some of the results concerning them were 
misstated in \cite{Sel6} and \cite{PerinElementary} (see \cite{PerinElementaryErratum}), which is why {\em extended} hyperbolic floors and 
towers had to be introduced.

Theorem 6 of \cite{Sel6} characterizes finitely generated models of the free group as hyperbolic towers. A corrected statement is 
\begin{theorem} \label{ElementaryFreeIFFHypTower}
Let $G$ be a finitely generated group. Then $G\models T_{fg}$ if and only if $G$ is an extended hyperbolic tower 
over the trivial subgroup.
\end{theorem}

\begin{remark} 
\label{ChangeInDefHypTower} One of the key steps in the proof of the ``only if'' direction of this result is to prove that from a map $G \to G$ preserving some of the structure of the group $G$, one can build a retraction $r: G \to G'$ to a proper subgroup which makes $(G, G', r)$ into a hyperbolic floor. 

However, in a few low complexity cases for $G$, this does not hold and the best we can get is a retraction making $(G, G', r)$ into an {\em extended} hyperbolic floor (see \cite{PerinElementaryErratum}).

This key step was made explicit in \cite{PerinElementary}, where it is stated as Proposition 5.11 and given a more detailed proof. However, in this paper too these counterexamples were overlooked. A corrected version of the proof this proposition can be found in \cite{PerinElementaryCorrected}. 
The ``if'' direction of \ref{ElementaryFreeIFFHypTower} in these the exceptional cases is not dealt with in \cite{Sel6}, but the proof can be extended in a straightforward way according to Sela \cite{SelaPrivate}.
\end{remark}

Sela also uses the notion of hyperbolic towers in \cite{Sel7} to classify torsion-free hyperbolic groups up to 
elementary equivalence. He shows that to every torsion free hyperbolic group $\Gamma$ can 
be associated a subgroup $C(\Gamma)$ which he calls its elementary core, over which $\Gamma$ admits a 
structure of hyperbolic tower, which is well defined up to isomorphism, and such that two torsion free 
hyperbolic groups are elementarily equivalent if and only if they have isomorphic elementary cores. 
He also shows that if $\Gamma$ is not elementarily equivalent to the free group, then $C(\Gamma)$ is an elementary subgroup of $\Gamma$. According to Sela \cite{SelaPrivate}, the proof of this last result can be adapted to give in fact 
\begin{theorem} \label{SubtowersAreElementary}
Suppose $\Gamma$ is a torsion free hyperbolic group which admits a structure of hyperbolic tower over a non abelian subgroup $H$. Then $H$ is an elementary subgroup of $G$.
\end{theorem}

The converse of this result is given by Theorem 1.2 of \cite{PerinElementary} so we get:
\begin{theorem} \label{ElementarySubgroups} Let $\Gamma$ be a torsion free hyperbolic group, and let $H$ be a non abelian subgroup of $\Gamma$.
Then $H$ is an elementary subgroup of $ \Gamma$ if and only if $\Gamma$ admits a structure of hyperbolic tower over $H$. 
\end{theorem}

The following result is Theorem 7.1 of \cite{PerinSklinosHomogeneity}. Before stating it we recall that 
the connectedness of $T_{fg}$ implies that $p_0$ is stationary (as any non forking extension of $p_0$ is also a generic type). Thus, following our discussion in Section \ref{Stable}, we denote by $p_0^{(k)}$ the type of $k$-independent realizations of $p_0$.

\begin{theorem} \label{ElementsOfPrimitiveTypeIFF} Let $G$ be a non abelian finitely generated group. 
Let $(u_1, \ldots, u_k)$ be a $k$-tuple of elements of $G$ for $k  \geq 1$. 

Then $(u_1, \ldots, u_k)$ realizes $p^{(k)}_0$ if and only if $H_u= \langle u_1, \ldots, u_k \rangle$ is 
free of rank $k$ and $G$ admits a structure of extended hyperbolic tower over $H_u$.
\end{theorem}

\subsection{Hyperbolic tower structures of the connected sum of four projective planes}

Admitting a structure of hyperbolic tower is quite a restrictive condition. For example we have 
\begin{lemma} \label{NoHypFloorInFreeGroup} If $\F$ is a free group, it does not admit any structure of extended hyperbolic floor over a subgroup.
\end{lemma}

\begin{proof} Lemma 5.19 in \cite{PerinElementary} states that free groups do not admit structures of hyperbolic floors. The argument given for the proof does not use the fact that surface type vertex groups have non abelian images by the retraction $r$, thus it can be applied to \textbf{extended} hyperbolic floors as well.
\end{proof}

Let $S_4$ denote the fundamental group of the surface $\Sigma_4$ which is the connected sum of four projective planes 
(i.e. the non orientable closed surface of characteristic $-2$). It has a trivial structure of hyperbolic tower over $\{1\}$. 

But this is not the only extended hyperbolic tower structure it admits. The following lemma gives some structures of extended hyperbolic floor for $S_4$. 

\begin{lemma} \label{FloorStructuresOfS4} Suppose $H$ is a non trivial subgroup of $S_4$, over which $S_4$ admits a structure of extended hyperbolic floor. Then $H$ is cyclic, and $S_4$ admits one of the following presentations
\begin{itemize}
 \item $\langle h, a, b, c \mid h^{2} = a^2 b^2 c^2 \rangle$;
 \item $\langle h, a, b, t \mid h tht^{-1} = a^2 b^2 \rangle$;
\end{itemize}
where $h$ generates $H$.
Conversely, given such a presentation, $S_4$ admits a structure of extended hyperbolic floor over the subgroup $H$ generated by $h$. 
\end{lemma}

These two structures are illustrated by Figure \ref{HypTowerStructuresOfS4}. In both pictures, the fundamental group of the space inside the box is 
$H = \langle h \rangle$. The fundamental group of the upper surface in the picture on the left is the subgroup generated by $a, b, c$ in $S_4$. In the picture on the right, the fundamental group of the upper surface is the subgroup of $S_4$ generated by $a$ and $b$.

\begin{figure}[!ht] 
\begin{center}
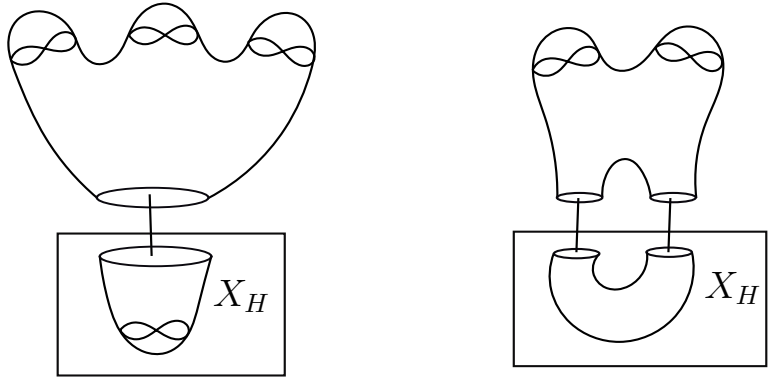
\caption{Hyperbolic floor structures of the connected sum of four projective planes.}
\label{HypTowerStructuresOfS4}
\end{center}
\end{figure}

\begin{proof} Suppose $S_4$ admits a structure of hyperbolic floor over a proper subgroup $H$, 
and denote by $\Gamma$ the associated graph of groups decomposition. By Lemma \ref{GOGSurfaceGroup}, 
$\Gamma$ is dual to a set of simple closed curves on $\Sigma_4$. In particular the surfaces of 
$\Gamma$ correspond to $\pi_1$-injected subsurfaces of $\Sigma_4$. Denote by $\Sigma$ the (possibly disconnected) 
subsurface of $\Sigma_4$ formed by all these subsurfaces, and by $\Sigma'$ the closure of its complement in 
$\Sigma_4$. We have $\chi(\Sigma) + \chi(\Sigma') = \chi(\Sigma_4) = -2$. Since the connected 
components of $\Sigma$ are surfaces of a hyperbolic floor decomposition, they are 
punctured tori or have characteristic at most $-2$. This implies in particular that $\Sigma$ can have at most two connected components. If it has exactly two, they must be punctured tori, and $\Sigma'$ has characteristic $0$ with two boundary components. Thus $\Sigma'$ must be a cylinder, but this contradicts the non orientability of $\Sigma_4$. So in fact, $\Sigma$ is connected, and there are only two possibilities: either $\chi(\Sigma) = -1$ and $\Sigma$ is a punctured torus, or $\chi(\Sigma) = -2$.

In the first case, $\Sigma'$ has one boundary component, characteristic $-1$, and is non orientable: 
it must be a punctured Klein bottle. However there cannot be a retraction of $S_4$ on the fundamental 
group of $S'$ of this punctured Klein bottle. Indeed, $S_4$ then admits a presentation of the form $\langle a, b, d_1, d_2 \mid [a,b]= d^2_1d^2_2 \rangle$ where $S'$ is the free subgroup of rank $2$ of $S_4$ generated by $d_1$ and $d_2$. If there exists a retraction $r: S_4 \to S'$, it fixes $d_1$ and $d_2$ thus $d^2_1d^2_2 = r([a,b])$ is a commutator, this is a contradiction. 

Thus the second alternative holds. This implies that $\chi(\Sigma')=0$. In particular, each connected component of 
$\Sigma'$ has characteristic $0$, hence must be a cylinder or a M\"obius band. Since $H$ is contained in a subgroup of 
$S_4$ corresponding to one of these connected components it must be cyclic, say $H = \langle h \rangle$. If $H$ is contained in a connected component of $\Sigma'$ which is a M\"obius band, its complement is a punctured connected sum of three projective planes, so $S_4$ admits a presentation as $\langle h, a, b, c \mid h^{2} = a^2 b^2 c^2 \rangle$. If $H$ is contained in a connected component of $\Sigma'$ which is a cylinder, its complement is a twice punctured Klein bottle, so $S_4$ admits a presentation as $\langle h, a, b, t \mid h tht^{-1} = a^2 b^2 \rangle$.

Let us now prove the converse. Let $\Gamma$ be the graph of groups decomposition for $S_4$ 
which consists of a single non surface type vertex with corresponding vertex group $\langle h \rangle$, 
and a single surface type vertex with corresponding group the subgroup generated by $\{a, b, c\}$ 
(respectively $\{a, b, h, tht^{-1}\}$) and corresponding surface a punctured connected sum of three projective planes 
(respectively a twice punctured Klein bottle). Consider the retraction 
$r': S_4 * \langle x \rangle \to \langle h \rangle * \langle x \rangle$ defined 
by $r'(a) = h$ and $r'(b) = r'(c^{-1})=x$ (respectively $r'(a) = r'(b^{-1})=x$ and $r'(t)=1$), and $r'(x)=x$. 
The conditions of Definition \ref{HypFloor} are satisfied by $\Gamma$ and $r'$, which proves the result.
\end{proof}

Noting that $S_4$ is freely indecomposable, one gets
\begin{corollary} \label{TowerStructuresOfS4} If $S_4$ admits a structure of hyperbolic tower over a subgroup $H$, then either $H$ is trivial or $H = \langle h \rangle$ is cyclic and $S_4$ admits one of the two presententations in terms of $h$ given in Lemma \ref{FloorStructuresOfS4}.
\end{corollary}

\section{Omitting, realizing and isolating types in the theory of free groups} \label{HypTowerSec} \label{IsolatingResultsSec}

We observe that Theorem \ref{MaxIndAreBases} 
implies that for $n \geq 2$, the free group $\F_n$ on $n$ generators realizes $p_0^{(n)}$ but omits $p_0^{(n+1)}$. 
Thus, it is natural to ask whether this holds also for $n=1$: is there a group $G$ which realizes $p_0$ yet omits $p_0^{(2)}$ ?

Pillay answered the above question in the affirmative in a non constructive way, using purely model theoretic methods. 
He then naturally asked whether an explicit model realizing $p_0$, but omitting $p_0^{(2)}$, exists. 
Such a group is exhibited in Proposition \ref{S4OmitsP02}, however we first give Pillay's elegant, but nonconstructive, argument.

His proof is based on the notion of semi-isolation, which we recall below, together with the following result (see \cite[Theorem 2.1]{SklinosGenericType}).
\begin{theorem}\label{P0Isolated} The generic type $p_0$ is not isolated in $T_{fg}$.
\end{theorem}

\begin{definition}
Let $\mathbb{M}$ be a big saturated model of a stable theory $T$, and let $\bar{a},\bar{b}$ be tuples in $\mathbb{M}$. 
The type $tp(\bar{a}/\bar{b})$ is semi-isolated, if there is a formula $\phi(\bar{x},\bar{y})$ (over the empty set) such that:
\begin{itemize}
  \item[(i)] $\mathbb{M}\models \phi(\bar{a},\bar{b})$;
  \item[(ii)] $\mathbb{M}\models \phi(\bar{x},\bar{b})\rightarrow tp(\bar{a})$. 
\end{itemize}
\end{definition}
 
 The following lemma connecting the notions of semi-isolation and forking will be useful (see \cite[Lemma 9.53(ii)]{PillayIntro}).
 \begin{lemma}\label{Fork}
 Suppose $tp(\bar{a}/\bar{b})$ is semi-isolated and $tp(\bar{a})$ is not isolated. Then $tp(\bar{a}/\bar{b})$ forks 
 over $\emptyset$.
 \end{lemma}

We are now ready to give Pillay's proof.

\begin{theorem}\label{P0squaredOverP0} \cite{PillayPrivate} There exists a group $G$ such that $G\models p_0$ and $G\not\models p_0^{(2)}$.
\end{theorem}
\begin{proof}
By the omitting types theorem (see \cite[Theorem 4.2.3,p.125]{MarkerModelTheory}), it is enough to prove that $p_0^{(2)}$ is not isolated in $p_0(c)$, i.e.
the complete theory in $\mathcal{L}=\{\cdot,^{-1},1,c\}$ axiomatized by $p_0(c)$. 
Note that if $\F_2=\langle e_1,e_2 \rangle$, then $(\F_2, e_1)$ is a model of $p_0(c)$. 

Suppose, for the sake of contradiction, that $\phi(x,y,c)$ isolates $p_0^{(2)}$ in $p_0(c)$. Let $(a,b)$ be 
a realization of $\phi(x,y,e_1)$ in $\F_2$. As $\F_2\models \phi(x,y,e_1)\rightarrow p_0^{(2)}$ 
we have by Theorem \ref{MaxIndAreBases} that $a,b$ form a basis of $\F_2$. In particular there is a word $w(x,y)$, such that $w(a,b)=e_2$. 
Now it is easy to see that the formula $\psi(z,u):=\exists x,y \left( \phi(x,y,u)\land z= w(x,y) \right)$ semi-isolates $tp(e_2/e_1)$. But 
as $tp(e_2)$ is not isolated, Lemma \ref{Fork} gives that $tp(e_2/e_1)$ forks over $\emptyset$, contradicting Theorem \ref{MaxIndAreBases}.
\end{proof}

We note that the above proof does not give much information about the group $G$, apart from the fact that it is countable.

Theorem \ref{P0Isolated} implies that there exists a model of $T_{fg}$ omitting the generic type. 
Using the results above, we can show that this model cannot be finitely generated. 

\begin{proposition}\label{P0InAllFGModels} Suppose $G$ is a finitely generated model of the theory $T_{fg}$ of non abelian free groups. 
Then $G$ realizes $p_0$.
\end{proposition}

\begin{proof} By Theorem \ref{ElementaryFreeIFFHypTower}, $G$ admits a structure of extended hyperbolic tower over $\{1\}$. 
The ground floor $G^m$ of this tower is a non trivial free product of a (possibly trivial) free group $\F$ with 
fundamental groups $S_1, \ldots, S_q$ of closed hyperbolic surfaces. By Theorem \ref{ElementsOfPrimitiveTypeIFF}, 
it is enough to show that $G$ has a structure of extended hyperbolic tower over a cyclic group $Z$.

We may assume $\F$ is trivial, since otherwise any cyclic free factor of $\F$ will do.

If $q$ is nonzero, $S_1$ admits a presentation as 
$$ \langle a_1, \ldots, a_g, b_1, \ldots, b_g \mid [a_1, b_1]\ldots[a_g,b_g] = 1 \rangle$$ 
if it is orientable, and
$$ \langle d_1, \ldots, d_p \mid d^2_1 \ldots d_p^2= 1 \rangle$$ if not. Let $H$ be the subgroup generated by $a_1, b_1$ 
in the first case, and $d_1, d_2$ in the second.

The map $r$ fixing $a_1, b_1$, sending $a_2$ to $b_1$, $b_2$ to $a_1$ and $a_j, b_j$ to $1$ for $j>2$ (respectively fixing $d_1, d_2$, 
sending $d_3$ to $d^{-1}_2$,  $d_4$ to $d^{-1}_1$ and $d_j$ to $1$ for $j>4$) is a retraction of $S_1$ onto the subgroup 
$H \simeq \F_2$, which we can extend into a retraction of $G^m$ onto $H * S_2 * \ldots * S_p$. 
However, this retraction makes $(G^m, r(G^m), r)$ an extended hyperbolic floor only if the surface corresponding 
to $a_2, b_2, \ldots, a_g, b_g$ (respectively $d_3, \ldots, d_p$) is a punctured torus or has characteristic at most $-2$. 
This fails to be the case only in the non orientable case and if $p=4$, that is, if $S_1$ is the connected sum of four 
projective planes. In all the other cases, we can take $Z$ to be any cyclic free factor of $H$ and the result is proved.

If $S_1$ is the fundamental group of the connected sum of four projective planes, 
choose a presentation of $S_1$ as $\langle h, a, b, c \mid h^{2} = a^2 b^2 c^2 \rangle$. If $q \geq 2$, 
we define a retraction $r: S_1 *  \ldots * S_q \to \langle h \rangle * S_2 * \ldots * S_q$ by $r(a) = h$, $r(b) = r(c^{-1}) = s$ 
for some non trivial element $s$ of $S_2$. Then $(G^m, r(G^m), r)$ is a hyperbolic floor. 
If $q=1$, we have seen in Lemma \ref{FloorStructuresOfS4} that $G^m$ admits a structure of extended hyperbolic floor over $\langle h \rangle$, 
so $G$ is an extended hyperbolic tower over $\langle h \rangle$.
\end{proof}

On the other hand, the following proposition 
gives an alternative proof of Theorem \ref{P0squaredOverP0}.

\begin{proposition} \label{S4OmitsP02}
Let $S_4$ be the fundamental group of the connected sum of four projective planes. Then $S_4$ omits $p^{(2)}_0$. 
\end{proposition}
\begin{proof}
Suppose, for the sake of contradiction, that $S_4\models p_0^{(2)}(u,v)$. Then by Theorem \ref{ElementsOfPrimitiveTypeIFF}, $u,v$ generate a free group $H$ of rank $2$ over which $S_4$ admits a structure of hyperbolic tower. 
By Corollary \ref{TowerStructuresOfS4} we know that no such structure exists.
\end{proof}

We conclude this section by giving another application of Corollary \ref{TowerStructuresOfS4}.

The following result is easily deduced from Proposition 5.9 and Proposition 6.2 of \cite{PerinSklinosHomogeneity}.
\begin{proposition}\label{EmbeddingOrHyperbolicFloor} 
Let $G$ and $G'$ be torsion-free hyperbolic groups. Let $\bar{u}$ and $\bar{v}$ be non trivial tuples of elements of $G$ and $G'$ respectively, and let $U$ be a finitely presented subgroup of $G$ which contains $\bar{u}$ and is freely indecomposable with respect to it.

If $\tp^G(\bar{u}) = \tp^{G'}(\bar{v})$, then either there exists an embedding $U \hookrightarrow G'$ which sends $\bar{u}$ to $\bar{v}$, or $U$ admits the structure of a hyperbolic floor over $\langle \bar{u} \rangle$.
\end{proposition}

We can now prove:

\begin{theorem}\label{Wagner}
Let $\bar{v}$ be a non trivial tuple of elements in a non abelian free group $\F$. Then $tp^{\F}(\bar{v})$ is not isolated.
\end{theorem}
\begin{proof}
Suppose $tp^{\F}(\bar{v})$ is isolated: then there exists a tuple $\bar{u}$ in $S_4$ such that $tp^{S_4}(\bar{u})=tp^{\F}(\bar{v})$. As 
$S_4$ does not embed in $\F$, Proposition \ref{EmbeddingOrHyperbolicFloor} applied to $U = S_4$ gives that $S_4$ is a hyperbolic floor 
over a subgroup $H$ containing $\langle \bar{u}\rangle$. 
Lemma \ref{FloorStructuresOfS4} implies that $H$ is a cyclic group, whose generator $h$ realizes $p_0$ by Theorem \ref{ElementsOfPrimitiveTypeIFF}. 
The tuple $\bar{u}$ is thus of the form $(h^{k_1}, \ldots, h^{k_n})$, and since its type is isolated, the type of $h^{k_1}$ is isolated: 
by a formula $\theta(x)$ say. Let $\psi(x)$ be a formula in $p_0$: by uniqueness of roots in $S_4$, the only $k_1$-th root of 
$h^{k_1}$ is $h$, so the formula 
$F(u) :\forall z \;(z^{k_1} = u \to \psi(z))$ is in the type of $h^{k_1}$. In particular
$$ S_4 \models \forall z \; (\theta(z^{k_1}) \to \psi(z)) $$
Thus $\theta(z^{k_1})$ isolates $p_0(z)$, which contradicts Theorem \ref{P0Isolated}.
\end{proof}

\section{Maximal independent sequences}\label{MaxIndepSeqSec}

The following result gives an example witnessing that $p_0$ has weight greater than $1$.

\begin{proposition}\label{Main} Let $S$ be the fundamental group of the orientable closed surface of characteristic $-2$, and let $G$ be the free product $\Z * S$.
 Then $G$ admits maximal independent sets of realizations of $p_0$ of cardinality $2$ and $3$.
\end{proposition}

\begin{proof}
We choose the following presentation for $G$:
$$\langle a, a', b, b',z \mid [a,b][a', b']=1 \rangle$$

The group $G$ admits at least three distinct hyperbolic tower structures.
\begin{enumerate}
\item {\bf The trivial structure:} $m = 0$ and $G = G^0 = \langle z \rangle * S$ is a free product of a free group and a closed surface group.

\item {\bf The structure over the subgroup} $H_1 = \langle a, b, z \mid  \rangle \simeq \F_3$: there is a hyperbolic floor $(G, H_1, r)$ described as follows.

The hyperbolic floor decomposition $\Lambda_1$ consists of 
\begin{itemize}
\item one vertex with vertex group $H_1$, 
\item one surface vertex with vertex group generated by $a', b'$ (the corresponding surface being a punctured torus).
\end{itemize}
The edge group is generated by $[a,b]$. The retraction $r: G \to H_1$ is given by $r(a)=a$, $r(b)=b$, $r(z)=z$, and $r(a')=b$, $r(b')=a$.

\item {\bf The structure over the subgroup} $H_2 = \langle a, za'z^{-1} \rangle \simeq \F_2$. There is a hyperbolic floor $(G, H_2, r)$ described as follows.

The hyperbolic floor decomposition $\Lambda_2$ consists of 
\begin{itemize}
\item one vertex with vertex group $H_2$, 
\item one surface vertex corresponding to a four times punctured sphere, whose maximal boundary subgroups are generated by $a$,$ba^{-1}b^{-1}$, $a'$, and $b'a'^{-1}b'^{-1}$ respectively.
\end{itemize}
The embeddings of the corresponding edge groups into $H_2$ send them on the subgroups generated by $a$, $a^{-1}$, $za'z^{-1}$, and $za'^{-1}z^{-1}$ respectively (so the Bass-Serre elements are $b$, $z^{-1}$, and $b'z^{-1}$).  

The retraction $r: G \to H_2$ is given by $r(a)=a$, $r(b)=1$, $r(a')=za'z^{-1}$, $r(b')=1$ and $r(z)=1$.
\end{enumerate}

\begin{figure}[!ht]
\begin{center} \scalebox{0.7}{
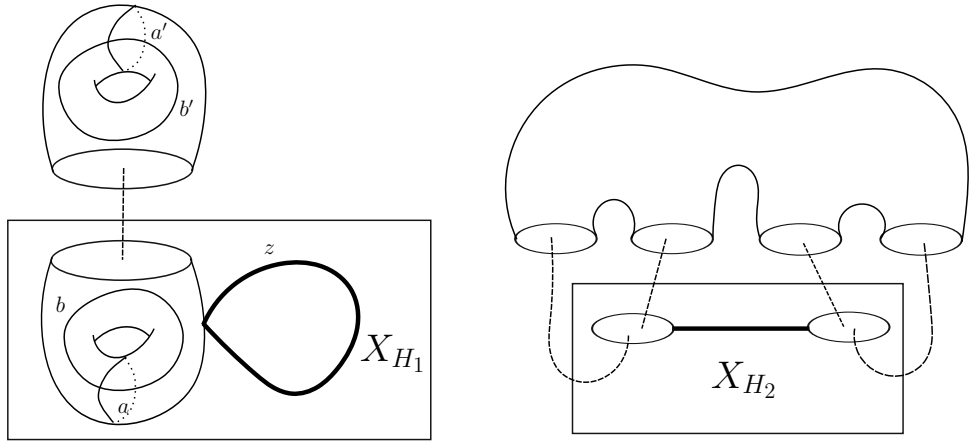}

\caption{The tower structures of $G$ over $H_1$ and $H_2$.}
\end{center}
\end{figure}

We claim that $G$ does not admit a tower structure over any rank $4$ subgroup $K_1$ in which $H_1$ is a free factor, nor over any rank $3$ subgroup $K_2$ in which $H_2$ is a free factor.

Indeed, suppose that such a subgroup $K_i$ exists for $i =1$ or $i=2$: in the associated graph of group decomposition $\Gamma_i$, the subgroup $H_i$ is elliptic, so the boundary subgroups of the surface group $S$ of $\Lambda_i$ are all elliptic. By Lemma \ref{GOGSurfaceGroup}, the induced graph of group decomposition for $S$ is dual to a set of non parallel simple closed curves on the punctured torus if $i=1$, or on the four times punctured sphere if $i=2$. In other words, the surfaces of $\Gamma_i$ are proper subsurfaces of the punctured torus or the four punctured sphere. But the only proper $\pi_1$-embedded subsurfaces that these surfaces admits are thrice punctured spheres or cylinders, and these are not permitted in tower structures.

Thus any basis for $H_1$ realizes $p_0^{(3)}$ in $G$, but cannot be extended to a realization of $p_0^{(4)}$, hence it is maximal. Similarly, any basis for $H_2$ is a maximal realization of $p^{(2)}_0$.
\end{proof}

\begin{remark} A similar proof would show that $S* \F_n$ admits maximal independent sets of realizations of $p_0$ of size both $n+1$ and $n+2$.
\end{remark}

In this low complexity case, it is easy to compute all the different possible tower structures that a given group may admit. In a more general setting, this becomes trickier. In particular, it would be interesting to find examples of finitely generated models of $T_{fg}$ (i.e. extended hyperbolic towers over the trivial subgroup) which admit maximal independent sets of realizations of $p_0$ of sizes whose ratio is arbitrarily large, thus witnessing directly infinite weight.

\section{Homogeneity and free products}\label{HomOfFreeProductsSec}

We mention one last application of the notion of hyperbolic towers: the free product of two homogeneous groups is not necessarily homogeneous.

For this we show:
\begin{lemma} \label{S2Homogeneous} Let $\Sigma$ denote the closed orientable surface of characteristic $-2$. 
The fundamental group $S$ of $\Sigma$ is homogeneous.
\end{lemma}
Note that Corollary 8.5 of \cite{PerinSklinosHomogeneity} states that the fundamental group of a surface of 
characteristic at most $-3$ is not homogeneous. 

\begin{proof} Suppose $\bar{u}$ and $\bar{v}$ are tuples which have the same type in $S$. 

Suppose first that there exist embeddings $i: S \hookrightarrow S$ and $j: S \hookrightarrow S$ such 
that $i(\bar{u}) = \bar{v}$ and $j(\bar{v})= \bar{u}$. Then since $S$ is freely indecomposable, 
relative co-Hopf property of torsion-free hyperbolic group (see Corollary 4.19 of \cite{PerinElementary}) 
implies that $j \circ i$ is an isomorphism, hence so are $i$ and $j$. 

We can now assume without loss of generality that there does not exist any embedding $i: S \hookrightarrow S$ such 
that $i(\bar{u}) = \bar{v}$. By Proposition \ref{EmbeddingOrHyperbolicFloor} applied to $U=S$, 
this implies that $S$ admits a structure of extended hyperbolic floor over a proper subgroup $U$ containing $\bar{u}$.

By Lemma \ref{GOGSurfaceGroup}, the hyperbolic floor decomposition $\Gamma$ is dual to a set ${\cal C}$ of 
disjoint simple closed curves on $\Sigma$. This decomposes $\Sigma$ into two (possibly disconnected) subsurfaces $\Sigma_0$ and $\Sigma_1$, 
corresponding respectively to non surface type vertices and surface type vertices. 
Now $\Sigma_0$ and $\Sigma_1$ satisfies $\chi(\Sigma_0)+ \chi(\Sigma_1) = -2$, and either $\Sigma_1$ is a punctured torus, or
it has characteristic at most $-2$.

In the first case, $\Sigma_0$ is also a punctured torus and $U = \pi_1(\Sigma_0)  \simeq \F_2$. In the second case, 
we get that $\chi(\Sigma_0)=0$ so $\Sigma_0$ must be a cylinder and $\bar{u}$ lies in $\pi_1(\Sigma_0)$. 
Let $\Sigma'_0$ be a punctured torus containing $\Sigma_0$: since $S$ admits a structure of hyperbolic 
floor over $\pi_1(\Sigma'_0)$ which contains $\bar{u}$, we may assume we are also in the first case. Thus $S$ admits the following presentation
$$\langle \alpha_0, \alpha_1, \beta_0, \beta_1 \mid [\alpha_0, \beta_0]=[\alpha_1, \beta_1] \rangle$$ 
with $U = \langle \alpha_0, \beta_0 \rangle$ and $\pi_1(\Sigma_1)= \langle \alpha_1, \beta_1 \rangle$ 

\paragraph{Case 1} If there does not exist any embedding $j: S \hookrightarrow S$ such that $j(\bar{v}) = \bar{u}$, 
we can deduce similarly that $S$ has a structure hyperbolic floor over a subgroup $V \simeq \F_2$ which contains $\bar{v}$, and that $S$ admits a presentation as $\langle \alpha'_0, \alpha'_1, \beta'_0, \beta'_1 \mid [\alpha'_0, \beta'_0]=[\alpha'_1, \beta'_1] \rangle$ with $V = \langle \alpha'_0, \beta'_0 \rangle$.

Note that if there exists an isomorphism $f: U \to V$ sending $\bar{u}$ to $\bar{v}$, we must have $f([\alpha_0, \beta_0]) = g[\alpha'_0, \beta'_0] g^{-1}$ for some $g$ in $V$ (in $\F_2$, all the commutators of two elements forming a basis are conjugate). Thus $f$ can be extended to an automorphism of $S$ by letting $f(\alpha_1) = g \alpha'_1 g^{-1}$ and $f(\beta_1) = g \beta'_1 g^{-1}$. We will now show that such an isomorphism $U \to V$ always exists.

If $U$ is freely indecomposable with respect to $\bar{u}$, by Lemma \ref{NoHypFloorInFreeGroup} and by Proposition \ref{EmbeddingOrHyperbolicFloor}, there is an embedding $f:U \hookrightarrow V$ sending $\bar{u}$ to $\bar{v}$. The smallest free factor of $V$ containing $\bar{v}$ contains $f(U)$, thus it cannot be cyclic. In particular, we have that $V$ is freely indecomposable with respect to $\bar{v}$. This implies in a similar way that there is an embedding $h:V \hookrightarrow U$ sending $\bar{v}$ to $\bar{u}$. Considering $h \circ f$ and using relative co-Hopf property for torsion-free hyperbolic groups shows $f$ is in fact an isomorphism, again proving the claim.

If $\bar{u}$ is contained in a cyclic free factor $\langle u_0 \rangle$ of $U$, then $\bar{v}$ is contained in a cyclic free factor $\langle v_0 \rangle$ of $V$. Then $\bar{u} = (u^{k_1}_0, \ldots, u^{k_l}_0)$, but since $\bar{u}$ and $\bar{v}$ have the same type, we have $\bar{v} = (v^{k_1}_0, \ldots, v^{k_l}_0)$. Thus we can easily find an isomorphism $f$ as required.

\paragraph{Case 2} Suppose now that there exists an embedding $j: S \hookrightarrow S$ such that $j(\bar{v}) = \bar{u}$.
 The hyperbolic floor decomposition $\Gamma$ of $S$ over $U$ (namely the amalgamated product $U *_{\langle c \rangle} S_1$) induces via $j$ a splitting of $S$ as a graph of groups with cyclic edge groups. By Lemma \ref{GOGSurfaceGroup}, this splitting is dual to a set ${\cal C}$ of simple closed curves on $\Sigma$.
 Since $\bar{u}$ is elliptic in the splitting $U *_{\langle c \rangle} S_1$, the tuple $\bar{v}$ is elliptic in this induced splitting. Thus $\bar{v}$ is contained in the fundamental group $S'_0$ of one of the connected components $\Sigma'_0$ of the complement in $\Sigma$ of ${\cal C}$, and $j(S'_0)$ is contained in $U$. 

We claim that $\Sigma'_0$ is a punctured torus, and that $j$ sends $S'_0$ isomorphically onto $U$ (as a surface group with boundary). This is enough to finish the proof, since we can then easily extend $j|_{S'_0}$ to an isomorphism $S \to S$. 

Let us thus prove the claim. The morphism $j$ is injective and sends elements corresponding to curves of ${\cal C}$ (in particular boundary subgroups of $S'_0$) to edge groups of $\Gamma$, that is, to conjugates of $\langle [\alpha_0, \beta_0] \rangle$. By Lemmas 3.10 and 3.12 of \cite{PerinElementary} we deduce that the complexity of $\Sigma'_0$ is at least that of $\Sigma_0$, and that if we have equality, then $j|_{S'_0}$ is an isomorphism of surface groups. In particular, if $\chi(\Sigma'_0) = -1$, $\Sigma'_0$ must have exactly one boundary component: hence it is a punctured tori, and the claim is proved. If $\chi(\Sigma'_0) = -2$, the surface $\Sigma'_0$ is a twice punctured tori. This implies that $S$ is generated by $S'_0$ together with an element $t$ which conjugates two maximal boundary subgroups $\langle d_1 \rangle$ and $\langle d_2 \rangle$ of $S'_0$ which are not conjugate in $S'_0$. Now $j(d_1)$ and $j(d_2)$ are conjugate in $S$, and both contained in $U$: they must be conjugate by an element $t'
$ of $U$ since $U$ is a retract of $S$. Now $j(t)^{-1}t'$ commutes to $j(d_1)$, so $j(t)^{-1}t'$, and thus $j(t)$, is contained in $U$. Finally $j(S) = \langle j(S'_0), j(t) \rangle \leq U$, but this is a contradiction since $U$ is free and $j$ is injective. 
\end{proof}

On the other hand, the following result is an immediate consequence of Proposition \ref{Main}:
\begin{lemma}
Let $G=\mathbb{Z}*S$. Then $G$ is not homogeneous.
\end{lemma}

Since $\Z$ is homogeneous, this gives an example of a free product of two homogeneous groups which fails to be homogeneous. 

\begin{proof} By Proposition \ref{Main}, there exists maximal realizations $(u_1, u_2)$ of $p^{(2)}_0$ and $(v_1, v_2, v_3)$ of $p^{(3)}_0$ in $G$. If $G$ was homogeneous, there would be an automorphism $\theta$ of $G$ sending $(v_1, v_2)$ to $(u_1, u_2)$ since they both realize $p^{(2)}_0$. But then $(u_1, u_2, \theta(v_3))$ would realize $p_0^{(3)}$, contradicting maximality of $(u_1, u_2)$. 
\end{proof}

\providecommand{\bysame}{\leavevmode\hbox to3em{\hrulefill}\thinspace}
\providecommand{\MR}{\relax\ifhmode\unskip\space\fi MR }
\providecommand{\MRhref}[2]{%
  \href{http://www.ams.org/mathscinet-getitem?mr=#1}{#2}
}
\providecommand{\href}[2]{#2}

\end{document}

%% file: GOG.pdf_tex

\begingroup
  \makeatletter
  \providecommand\color[2][]{%
    \errmessage{(Inkscape) Color is used for the text in Inkscape, but the package 'color.sty' is not loaded}
    \renewcommand\color[2][]{}%
  }
  \providecommand\transparent[1]{%
    \errmessage{(Inkscape) Transparency is used (non-zero) for the text in Inkscape, but the package 'transparent.sty' is not loaded}
    \renewcommand\transparent[1]{}%
  }
  \providecommand\rotatebox[2]{#2}
  \ifx\svgwidth\undefined
    \setlength{\unitlength}{466.4074832pt}
  \else
    \setlength{\unitlength}{\svgwidth}
  \fi
  \global\let\svgwidth\undefined
  \makeatother
  \begin{picture}(1,0.29941476)%
    \put(0,0){\includegraphics[width=\unitlength]{GOG.pdf}}%
    \put(0,0.00664553){\color[rgb]{0,0,0}\makebox(0,0)[lb]{\smash{\Large{$G_{v_2}$}}}}%
    \put(0.26520282,0.11141563){\color[rgb]{0,0,0}\makebox(0,0)[lb]{\smash{\Large{$G_{v_1}$}}}}%
    \put(0.29928536,0.00664553){\color[rgb]{0,0,0}\makebox(0,0)[lb]{\smash{\Large{$G_{v_3}$}}}}%
    \put(0.11605826,0.10182057){\color[rgb]{0,0,0}\rotatebox{20.53151568}{\makebox(0,0)[lb]{\smash{$G_{e_2}$}}}}%
    \put(0.16332632,0.03085172){\color[rgb]{0,0,0}\makebox(0,0)[lb]{\smash{$G_{e_3}$}}}%
    \put(0.2338295,0.25366828){\color[rgb]{0,0,0}\rotatebox{1.98180188}{\makebox(0,0)[lb]{\smash{$G_{e_1}$}}}}%
    \put(0.43233211,0.07224316){\color[rgb]{0,0,0}\makebox(0,0)[lb]{\smash{\Large{$X_{v_2}$}}}}%
    \put(0.79908596,0.05075913){\color[rgb]{0,0,0}\makebox(0,0)[lb]{\smash{\Large{$X_{v_3}$}}}}%
    \put(0.58494495,0.15425664){\color[rgb]{0,0,0}\rotatebox{18.718279}{\makebox(0,0)[lb]{\smash{$X_{e_2}$}}}}%
    \put(0.61794537,0.02078606){\color[rgb]{0,0,0}\makebox(0,0)[lb]{\smash{$X_{e_3}$}}}%
    \put(0.75693974,0.28189818){\color[rgb]{0,0,0}\rotatebox{1.79315156}{\makebox(0,0)[lb]{\smash{$X_{e_1}$}}}}%
    \put(0.73560528,0.15939696){\color[rgb]{0,0,0}\makebox(0,0)[lb]{\smash{\Large{$X_{v_1}$}}}}%
  \end{picture}%
\endgroup

%% file: GOGWithSurfaces.pdf_tex

\begingroup
  \makeatletter
  \providecommand\color[2][]{%
    \errmessage{(Inkscape) Color is used for the text in Inkscape, but the package 'color.sty' is not loaded}
    \renewcommand\color[2][]{}%
  }
  \providecommand\transparent[1]{%
    \errmessage{(Inkscape) Transparency is used (non-zero) for the text in Inkscape, but the package 'transparent.sty' is not loaded}
    \renewcommand\transparent[1]{}%
  }
  \providecommand\rotatebox[2]{#2}
  \ifx\svgwidth\undefined
    \setlength{\unitlength}{524.18783508pt}
  \else
    \setlength{\unitlength}{\svgwidth}
  \fi
  \global\let\svgwidth\undefined
  \makeatother
  \begin{picture}(1,0.32012233)%
    \put(0,0){\includegraphics[width=\unitlength]{GOGWithSurfaces.pdf}}%
    \put(0.15331859,0.00649941){\color[rgb]{0,0,0}\makebox(0,0)[lb]{\smash{\Large{$G_{v_2}$}}}}%
    \put(0.34954049,0.1493054){\color[rgb]{0,0,0}\makebox(0,0)[lb]{\smash{\Large{$G_{v_3}$}}}}%
    \put(0.44329099,0.00431918){\color[rgb]{0,0,0}\makebox(0,0)[lb]{\smash{\Large{$G_{v_4}$}}}}%
    \put(0.12824579,0.25286698){\color[rgb]{0,0,0}\makebox(0,0)[lb]{\smash{\Large{$G_{v_1}= \pi_1(\Sigma)$}}}}%
    \put(0.63842279,0.15693623){\color[rgb]{0,0,0}\makebox(0,0)[lb]{\smash{\Large{$G_{v_5}= \pi_1(\Sigma')$}}}}%
    \put(0.14968356,0.10859722){\color[rgb]{0,0,0}\rotatebox{-52.60467264}{\makebox(0,0)[lb]{\smash{$G_{e_1}$}}}}%
    \put(0.24659445,0.08328811){\color[rgb]{0,0,0}\rotatebox{18.718279}{\makebox(0,0)[lb]{\smash{$G_{e_2}$}}}}%
    \put(0.30157515,0.02939201){\color[rgb]{0,0,0}\makebox(0,0)[lb]{\smash{$G_{e_3}$}}}%
    \put(0.41114827,0.10620615){\color[rgb]{0,0,0}\rotatebox{-46.63277042}{\makebox(0,0)[lb]{\smash{$G_{e_4}$}}}}%
    \put(0.51395356,0.0686711){\color[rgb]{0,0,0}\rotatebox{56.72985358}{\makebox(0,0)[lb]{\smash{$G_{e_6}$}}}}%
    \put(0.41867253,0.1600927){\color[rgb]{0,0,0}\rotatebox{31.76980265}{\makebox(0,0)[lb]{\smash{$G_{e_5}$}}}}%
  \end{picture}%
\endgroup

%% file: hyptowerex.pdf_tex

\begingroup
  \makeatletter
  \providecommand\color[2][]{%
    \errmessage{(Inkscape) Color is used for the text in Inkscape, but the package 'color.sty' is not loaded}
    \renewcommand\color[2][]{}%
  }
  \providecommand\transparent[1]{%
    \errmessage{(Inkscape) Transparency is used (non-zero) for the text in Inkscape, but the package 'transparent.sty' is not loaded}
    \renewcommand\transparent[1]{}%
  }
  \providecommand\rotatebox[2]{#2}
  \ifx\svgwidth\undefined
    \setlength{\unitlength}{128.46875pt}
  \else
    \setlength{\unitlength}{\svgwidth}
  \fi
  \global\let\svgwidth\undefined
  \makeatother
  \begin{picture}(1,0.83945512)%
    \put(0,0){\includegraphics[width=\unitlength]{hyptowerex.pdf}}%
    \put(0.04946141,0.09317952){\color[rgb]{0,0,0}\makebox(0,0)[lb]{\smash{\tiny{$X_H$}}}}%
  \end{picture}%
\endgroup

%% file: HypTowerStrusturesOfS4.pdf_tex

\begingroup
  \makeatletter
  \providecommand\color[2][]{%
    \errmessage{(Inkscape) Color is used for the text in Inkscape, but the package 'color.sty' is not loaded}
    \renewcommand\color[2][]{}%
  }
  \providecommand\transparent[1]{%
    \errmessage{(Inkscape) Transparency is used (non-zero) for the text in Inkscape, but the package 'transparent.sty' is not loaded}
    \renewcommand\transparent[1]{}%
  }
  \providecommand\rotatebox[2]{#2}
  \ifx\svgwidth\undefined
    \setlength{\unitlength}{315.33372999pt}
  \else
    \setlength{\unitlength}{\svgwidth}
  \fi
  \global\let\svgwidth\undefined
  \makeatother
  \begin{picture}(1,0.44779485)%
    \put(0,0){\includegraphics[width=\unitlength]{HypTowerStrusturesOfS4.pdf}}%
    \put(0.24481739,0.08400391){\color[rgb]{0,0,0}\makebox(0,0)[lb]{\smash{\Large{$X_H$}}}}%
    \put(0.82875135,0.09406858){\color[rgb]{0,0,0}\makebox(0,0)[lb]{\smash{\Large{$X_H$}}}}%
  \end{picture}%
\endgroup

%% file: Lars.pdf_tex

\begingroup
  \makeatletter
  \providecommand\color[2][]{%
    \errmessage{(Inkscape) Color is used for the text in Inkscape, but the package 'color.sty' is not loaded}
    \renewcommand\color[2][]{}%
  }
  \providecommand\transparent[1]{%
    \errmessage{(Inkscape) Transparency is used (non-zero) for the text in Inkscape, but the package 'transparent.sty' is not loaded}
    \renewcommand\transparent[1]{}%
  }
  \providecommand\rotatebox[2]{#2}
  \ifx\svgwidth\undefined
    \setlength{\unitlength}{514.95527182pt}
  \else
    \setlength{\unitlength}{\svgwidth}
  \fi
  \global\let\svgwidth\undefined
  \makeatother
  \begin{picture}(1,0.45662391)%
    \put(0,0){\includegraphics[width=\unitlength]{Lars.pdf}}%
    \put(0.36698054,0.08302986){\color[rgb]{0,0,0}\makebox(0,0)[lb]{\smash{\huge{$X_{H_1}$}}}}%
    \put(0.73206079,0.05528824){\color[rgb]{0,0,0}\makebox(0,0)[lb]{\smash{\huge{$X_{H_2}$}}}}%
    \put(0.17944693,0.33936285){\color[rgb]{0,0,0}\makebox(0,0)[lb]{\smash{\large{$b'$}}}}%
    \put(0.14948595,0.42036847){\color[rgb]{0,0,0}\makebox(0,0)[lb]{\smash{\large{$a'$}}}}%
    \put(0.1150863,0.02865625){\color[rgb]{0,0,0}\makebox(0,0)[lb]{\smash{\large{$a$}}}}%
    \put(0.05072563,0.13407457){\color[rgb]{0,0,0}\makebox(0,0)[lb]{\smash{\large{$b$}}}}%
    \put(0.26711058,0.19510621){\color[rgb]{0,0,0}\makebox(0,0)[lb]{\smash{\large{$z$}}}}%
  \end{picture}%
\endgroup